\documentclass[11pt]{article}
\usepackage{amsmath, amssymb, amsthm}
\usepackage{verbatim}
\usepackage{multicol}
\usepackage{enumerate}
\usepackage{comment}
\usepackage[none]{hyphenat}
\usepackage{hyperref}
\hypersetup{
	colorlinks=true,
	linkcolor=blue,
	filecolor=magenta,
	urlcolor=cyan,
	citecolor=blue
}

\usepackage{pgf}
\usepackage{tikz}
\usetikzlibrary{math,calc,intersections}
\usetikzlibrary{positioning,arrows,shapes,decorations.markings,decorations.pathreplacing,matrix,patterns}
\tikzstyle{vertex}=[circle,draw=black,fill=black,inner sep=0,minimum size=3pt,text=white,font=\footnotesize]
\usepackage{cleveref}

\date{}
\title{\vspace{-0.8cm} String graphs have the Erd\H{o}s-Hajnal property}
\author{
	Istv\'{a}n Tomon\thanks{ETH Zurich. \emph{e-mail}: \textbf{istvan.tomon@math.ethz.ch}, Research supported by SNSF grant 200021-149111.}
}

\theoremstyle{plain}
\newtheorem{theorem}{Theorem}

\newtheorem{lemma}[theorem]{Lemma}

\theoremstyle{definition}

\begin{document}
\maketitle

\sloppy

\begin{abstract}
	A \emph{string graph} is the intersection graph of curves in the plane. We prove that there exists an absolute constant $c>0$ such that if $G$ is a string graph on $n$ vertices, then $G$ contains either a clique or an independent set of size at least $n^{c}$.
\end{abstract}

\section{Introduction}
Let $\alpha(G)$ and $\omega(G)$ denote the independence number and clique number of a graph $G$, respectively. A family of graphs is \emph{hereditary} if it is closed under taking induced subgraphs. Say that a family of graphs $\mathcal{G}$ has the \emph{Erd\H{o}s-Hajnal property}, if there exists a constant $c=c(\mathcal{G})>0$ such that each $G\in \mathcal{G}$ contains either a clique or an independent set of size at least $|V(G)|^{c}$. The celebrated conjecture of Erd\H{o}s and Hajnal \cite{EH89} states that if $\mathcal{G}$ is a hereditary family of graphs that is not the family of all graphs, then $\mathcal{G}$ has the Erd\H{o}s-Hajnal property (originally, this conjecture is stated such that for every graph $H$, the family of graphs not containing $H$ as an induced subgraph has the Erd\H{o}s-Hajnal property, but this is equivalent to the previous formulation). This conjecture is mostly wide open, and it has been verified for only certain special families $\mathcal{G}$, see the survey of Chudnovsky \cite{Ch} for a general reference.

Erd\H{o}s-Hajnal type questions are also extensively studied in geometric settings. The \emph{intersection graph} of a family of geometric objects $\mathcal{C}$ is the graph whose vertices correspond to the elements of $\mathcal{C}$, and two vertices are joined by an edge if the corresponding sets have a nonempty intersection. Perhaps one of the first geometric Erd\H{o}s-Hajnal type results is the classical folklore that if $G$ is the intersection graph of a family of $n$ intervals, then $G$ is perfect, so $G$ contains either a clique or an independent set of size at least $n^{1/2}$. It was proved by Larman et al. \cite{LMPT94} that if $G$ is the intersection graph of axis-parallel boxes in the plane, then $\max\{\alpha(G),\omega(G)\}=\Omega(\sqrt{\frac{n}{\log n}}).$ They also proved that if $G$ is the intersection graph of convex sets, then either $G$ or its complement contains a clique of size at least $n^{1/5}$. While intersection graphs of intervals, rectangles, disks, convex sets, etc. were the subjects of study in the early and mid 20th century, in the past few decades the focus shifted to more general geometric graphs. One such generalization is semi-algebraic graphs, that is, graphs whose vertices correspond to points in $\mathbb{R}^{d}$, and the edges are defined by polynomial relations (for precise definitions, see \cite{APPRS}). Indeed, intersection graphs of intervals, disks, and rectangles are special instances of semi-algebraic graphs. However, intersection graphs of convex sets are not. In general, it was proved by Alon et al. \cite{APPRS} that if $\mathcal{G}$ is a family of semi-algebraic graphs of bounded complexity, then $\mathcal{G}$ has the Erd\H{o}s-Hajnal property.

Another generalization, which is the main interest of this paper, is string graphs. A \emph{curve} (or a \emph{string}) is the image of a continuous function $\phi: [0,1]\rightarrow \mathbb{R}^2$. A \emph{string graph} is the intersection graph of a family of curves.  String graphs were introduced by Benzer \cite{B59} in 1959 to study topological properties of genetic structures, and later Sinden \cite{S66} considered such graphs to model printed circuits. Since, combinatorial and computational properties of string graphs are extensively studied. Note that, in certain sense, curves are the most general geometric objects on the plane one can consider: indeed, when talking about geometric objects, one of the weakest geometric properties one should require is arcwise connectedness, and any arcwise connected set on the plane can be approximated arbitrarily closely by curves. In particular, all of the aforementioned intersection graphs are string graphs as well.

 The question whether the family of string graphs have the Erd\H{o}s-Hajnal property is one of the central problems in the area, and was settled in a number of interesting special cases. See Alon et al. \cite{APPRS} for one of the first appearances of this problem, and \cite{FP08_survey} for a survey type paper on the topic.  Larman et al. \cite{LMPT94} proved that if $G$ is the intersection graph of $n$ \emph{$x$-monotone curves} (a curve is $x$-monotone if every vertical line intersects it in at most 1 point), then $\max\{\alpha(G),\omega(G)\}\geq n^{1/5}$. Also, it was proved by Fox, Pach and T\'oth \cite{FPT11} that for every $k$ there exists $c_k>0$ such that if $G$ is the intersection graph of $n$ curves, and any two curves intersect in at most $k$ points, then $G$ contains either a clique or an independent set of size at least $n^{c_k}$. In general, Fox and Pach \cite{FP14_sep2} proved the slightly weaker result that if $G$ is a string graph, then either $G$ or its complement contains a clique of size $n^{\Omega(1/\log \log n)}$. The main result of our paper is that the family of string graphs has the Erd\H{o}s-Hajnal property, which implies the Erd\H{o}s-Hajnal property of all of the aforementioned families of intersection graphs.

\begin{theorem}\label{thm:main}
	There exists an absolute constant $c>0$ such that for every positive integer $n$, if $G$ is a string graph with $n$ vertices, then $G$ contains either a clique or an independent set of size at least $n^{c}$.
\end{theorem}

Our proof of Theorem \ref{thm:main} closely follows the path laid out by the works of Fox and Pach \cite{FP10_sep1,FP12_inc,FP14_sep2}. In the next subsection, we discuss their ideas and outline our proof strategy. We also introduce our notation, which is mostly conventional.

\section{Overview of the proof}

Given a graph $G$ and two subsets $A$ and $B$ of $V(G)$, say that \emph{$A$ is complete to $B$} if every $a\in A$ and $b\in B$ is joined by an edge.

A graph $G$ is a \emph{comparability graph} if there exists a partial ordering $\prec$ on $V(G)$ such that for any $v,w\in V(G)$, we have $v\prec w$ or $w\prec v$ if and only if $vw$ is an edge of $G$. Also, $G$ is an \emph{incomparability graph}, if it is the complement of a comparability graph.

\bigskip

\noindent
\textbf{Previous approach:}  It turns out that string graphs and incomparability graphs are closely related. Indeed, it was proved by Lov\'asz \cite{Lo} and reproved in \cite{GRU83,PT06} that every incomparability graph is a string graph. On the other hand, Fox and Pach \cite{FP12_inc} proved that every dense string graph contains a dense incomparability graph on the same vertex set as a subgraph. This result is going to be one of the main ingredients of our proof of Theorem \ref{thm:main}, see Section \ref{sect:strings} for more details.

 A family of graphs $\mathcal{G}$ has the \emph{strong-Erd\H{o}s-Hajnal property}, if there exists a constant $b=b(\mathcal{G})>0$ such that the vertex set of every $G\in \mathcal{G}$ contains two disjoint subsets $A$ and $B$ such that $|A|=|B|\geq b|V(G)|$, and either there are no edges between $A$ and $B$, or $A$ is complete to $B$. It is not hard to show by a recursive argument that in a hereditary graph family, the strong-Erd\H{o}s-Hajnal property implies the Erd\H{o}s-Hajnal property, see \cite{APPRS}. One approach to proving Theorem \ref{thm:main} would be to show that the family of string graphs has the strong-Erd\H{o}s-Hajnal property. Indeed, with the exception of intersection graphs of $x$-monotone curves, every family of intersection graphs where the Erd\H{o}s-Hajnal property is known also has the strong-Erd\H{o}s-Hajnal property.

Unfortunately, the family of string graphs does not have the strong-Erd\H{o}s-Hajnal property, but it has something close to it. Let $G$ be a string graph on $n$ vertices. A separator theorem of Lee \cite{L17} shows that if $G$ is sufficiently sparse (meaning that $|E(G)|\leq \lambda n^2$ for some small constant $\lambda$), then $V(G)$ contains two linear sized subsets with no edges between them. This result is going to be another important ingredient in our proof, see Section \ref{sect:strings} for more details. On the other hand, by a result of Fox \cite{F06}, every dense incomparability graph on $n$ vertices contains two disjoint sets $A$ and $B$ of size $\Omega(\frac{n}{\log n})$ such that $A$ is complete to $B$, and this bound is the best possible up to the constant factor. But then remembering that every dense string graph contains a dense incomparability graph, we get that if $G$ is dense, then $G$ contains  two disjoint sets $A$ and $B$ of size $\Omega(\frac{n}{\log n})$ such that $A$ is complete to $B$. Therefore, one can conclude the following ``almost-strong-Erd\H{o}s-Hajnal property'': 

\begin{theorem}(\cite{L17,FP12_inc})
If $G$ is a string graph on $n$ vertices, then $V(G)$ contains two disjoint sets $A$ and $B$ such that either 
\begin{enumerate}
	\item $|A|=|B|=\Omega(n)$ and there are no edges between $A$ and $B$, or
	\item $|A|=|B|=\Omega(\frac{n}{\log n})$ and $A$ is complete to $B$.
\end{enumerate}
\end{theorem}

Again, these bounds are the best possible up to the constant factor. Then, by a recursive argument this theorem implies that if $G$ is a string graph on $n$ vertices, then $G$ contains a clique or an independent set of size $n^{\Omega(1/\log\log n)}$, see \cite{FP14_sep2}.

\bigskip

\noindent
\textbf{New ideas:} In order to improve this bound, we do the following. Instead of proving the strong-Erd\H{o}s-Hajnal property, we prove something slightly weaker which we call the \emph{quasi-Erd\H{o}s-Hajnal property}. Roughly, a family of graphs $\mathcal{G}$ has this property if for every $G\in\mathcal{G}$ there exist some $t\geq 2$ and $t$ disjoint subsets $X_1,\dots,X_t$ of $V(G)$ such that $|X_1|,\dots,|X_t|$ are ``large'' with respect to $t$ and $|V(G)|$, and either there are no edges between $X_i$ and $X_j$ for $1\leq i<j\leq t$, or $X_i$ is complete to $X_j$ for $1\leq i<j\leq t$. It turns out  that in hereditary families, this quasi-Erd\H{o}s-Hajnal property is equivalent to the Erd\H{o}s-Hajnal property, see Section \ref{sect:EH} for more details. Then, our main contribution to the proof of Theorem \ref{thm:main} is that in every dense incomparability graph $G$, there exist $t\geq 2$ and $t$ disjoint subsets $X_1,\dots,X_t$ such that $|X_1|,\dots,|X_t|$ are ``large'' with respect to $t$ and $|V(G)|$, and $X_i$ is complete to $X_j$ for $1\leq i<j\leq t$. This can be found in Section \ref{sect:poset}. But then, together with the aforementioned results of Lee \cite{L17} and Fox and Pach \cite{FP12_inc}, this implies that the family of string graphs has the quasi-Erd\H{o}s-Hajnal property.

\bigskip

 \noindent
 \textbf{Notation:} In the rest of our paper, we use the following standard graph theoretic notations. If $G$ is a graph, $\Delta(G)$ denotes the maximum degree of $G$, and if $v\in V(G)$, then $N(v)=\{w\in V(G):vw\in E(G)\}$ is the neighborhood of $v$. If $U$ is a subset of the vertex set, then $G[U]$ is the subgraph of $G$ induced on $U$. Given a poset $P$ with partial ordering $\prec$, a total ordering $<_l$ is a \emph{linear extension} of $\prec$ if $x\prec y$ implies $x<_l y$ for every $x,y\in P$. It is well known that every partial ordering has a linear extension (which might not be unique). Also, if $A,B\subset P$, then we write $A<_l B$ if $a<_l b$ for every $a\in A, b\in B$. We omit floors and ceilings whenever they are not crucial.

\section{Incomparability graphs}\label{sect:poset}

Our main contribution to the proof of Theorem \ref{thm:main} is the following result about partial orders, which might be of independent interest.

\begin{theorem}\label{thm:incomparability}
	For every $\alpha>0$ there exists $c>0$ such that the following holds. Let $G$ be an incomparability graph with $n$ vertices and at least $\alpha\binom{n}{2}$ edges. Then there exist $t\geq 2$ and $t$ disjoint subsets $X_1,\dots,X_t$ of $V(G)$ such that $X_i$ is complete to $X_j$ for $1\leq i<j\leq t$, and $(\frac{n}{|X_i|})^{c}<t$ for $i=1,\dots,t$. 
\end{theorem}

We would like to emphasize that $t$ depends on the incomparability graph $G$. In order to prove this theorem, it is slightly better to work with comparability graphs instead of incomparability graphs, so we prove the following equivalent statement instead.

\begin{theorem}\label{thm:poset}
	For every $\alpha>0$ there exists $c>0$ such that the following holds. Let $P$ be a comparability graph with $n$ vertices and at most $(1-\alpha)\binom{n}{2}$ edges. Then there exist $t\geq 2$ and $t$ disjoint subsets $X_1,\dots,X_t$ of $V(P)$ such that there are no edges between $X_i$ and $X_j$ for $1\leq i<j\leq t$, and $(\frac{n}{|X_i|})^{c}<t$ for $i=1,\dots,t$. 
\end{theorem}

 Let us prove this theorem. Instead of working with very dense comparability graphs, we would like to work with sparse ones. With the help of a technical lemma, we show that $P$ contains either a sparse comparability graph with linearly many vertices, or we can find $2$ linear sized subsets $X_1$ and $X_2$ with no edges between them. In order to show this, we make use of the following well known result.

\begin{lemma}\label{lemma:maxdeg}
	For every $\alpha>0$ there exists $\alpha_1$ such that the following holds. If $G$ is a graph with $n$ vertices and at most $(1-\alpha)\binom{n}{2}$ edges, then  $G$ contains an induced subgraph $G'$ such that $|V(G')|\geq \alpha_1 n$ and $\Delta(G')\leq (1-\alpha_1)|V(G')|$.
\end{lemma}

\begin{proof}
	By taking complements, this immediately follows from the following well known statement: There exists $\alpha_1>0$ such that if $G$ is a graph with $n$ vertices with at least $\alpha\binom{n}{2}$ edges, then  $G$ contains an induced subgraph $G'$ such that $|V(G')|\geq \alpha_1 n$ and the minimum degree of $G'$ is at least $\alpha_1|V(G')|$.
\end{proof}

\begin{lemma}\label{lemma:sparse}
	For every $\alpha,\epsilon>0$ there exists $\beta>0$ such that the following holds. Let $P$ be a comparability graph with $n$ vertices and at most $(1-\alpha)\binom{n}{2}$ edges, and let $<_l$ be a linear extension of the underlying partial order. Then either
	\begin{enumerate}
		\item there exist two disjoint subsets $A$ and $B$ of $V(P)$ such that $A<_l B$, $|A|=|B|\geq\beta n$, $|N(v)\cap B|\leq \epsilon |B|$ for every $v\in A$, and  $|N(w)\cap A|\leq \epsilon |A|$ for every $w\in B$,  or
		\item there exist two disjoint sets $X_1,X_2\subset V(P)$ such that there are no edges between $X_1$ and $X_2$, and $|X_1|,|X_2|\geq \beta n$.
	\end{enumerate} 
\end{lemma}

\begin{proof}
	By Lemma \ref{lemma:maxdeg}, there exists $\alpha_1$ (depending only on $\alpha$) such that $P$ contains an induced subgraph $P'$ with $n'=|V(P')|\geq \alpha_1 n$ and $\Delta(P')\leq (1-\alpha_1)n'$. We show that $\beta=\frac{\alpha_1^2\epsilon}{24}$ suffices.  Suppose that 2. does not hold, that is, there exist no two disjoint sets $X_1$ and $X_2$ such that there are no edges between $X_1$ and $X_2$, and $|X_1|,|X_2|\geq \beta n$. Then, we prove that 1. holds.
	
	 Let $\prec$ be the partial ordering of the underlying poset of $P$. Let $T$ be the $\frac{\alpha_1}{6}n'$ largest elements of $P'$ with respect to the linear extension $<_l$, and let $S=P'\setminus T$. Also, let $U$ be the $\frac{\epsilon}{4}|T|$ largest elements of $T$ with respect to $<_l$. Say that a vertex $v\in S$ is \emph{heavy} if $|N(v)\cap T|\geq \frac{\epsilon}{2}|T|$. Setting $X_1=N(v)\cap (T\setminus U)$ and $X_2=U\setminus N(v)$, there are no edges between $X_1$ and $X_2$. Indeed, otherwise, if $x\in X_1$ and $y\in X_2$ are joined by an edge, then $x<_l y$ implies $x\prec y$, and as $v\prec x$, we get that $v\prec y$, contradicting $y\not\in N(v)$. But if $v$ is heavy, then $|X_1|\geq \frac{\epsilon}{2} |T|-|U|=\frac{\epsilon}{4}|T|=\beta n$, so we must have $|X_2|<\beta n= \frac{\alpha_1}{3}|U|$. This implies $|N(v)\cap U|\geq (1-\frac{\alpha_1}{3}) |U|$. If $H$ is the set of heavy vertices, then the number of edges between $H$ and $U$ is at least $(1-\frac{\alpha_1}{3})|U||H|$, which implies that there exists a vertex in $U$ of degree at least $(1-\frac{\alpha_1}{3})|H|$. Therefore, by the maximum degree condition we can write $(1-\frac{\alpha_1}{3})|H|\leq(1-\alpha_1)n'$, which gives $|H|\leq \frac{1-\alpha_1}{1-\alpha_1/3}n'<(1-\frac{\alpha_1}{3})n'$. But then $|S\setminus H|=n'-|T|-|H|>\frac{\alpha_1}{6}n'$, so there are at least $\frac{\alpha_1}{6}n'$ vertices $v\in S$ such that $|N(v)\cap T|\leq \frac{\epsilon}{2} |T|$. Let $A$ be a set of $\frac{\alpha_1}{12}n'\geq \beta n$ such vertices. The number of edges between $A$ and $T$ is at most $\frac{\epsilon}{2}|A||T|$, so the number of vertices $v\in T$ such that $|N(v)\cap T|\geq \epsilon |A|$ is at most $|T|/2\leq \frac{\alpha_1}{12}n'$. Delete all such vertices from $T$, and perhaps some more, to get a set $B$ of size $\frac{\alpha_1}{12}n'$. Then $A$ and $B$ satisfy 1.
\end{proof}

Most of the work needed to prove Theorem \ref{thm:poset} is put into the following lemma.

\begin{lemma}\label{lemma:main}
	There exist positive real numbers $\epsilon$ and $\delta$ such that the following holds. Let $P$ be a comparability graph on $2n$ vertices, and let $<_l$ be a linear extension of the underlying poset. Let $A$ be the smallest $n$ elements of $P$ with respect to $<_l$, let $B=P\setminus A$, and suppose that $|N(v)\cap B|\leq \epsilon n$ for every $v\in A$ and $|N(w)\cap A|\leq \epsilon n$ for every $w\in B$. Then there exist $t\geq 2$ and $t$ disjoint sets $X_1,\dots,X_t\subset V(P)$ such that there are no edges between $X_i$ and $X_j$ for $1\leq i<j\leq t$, and $\delta(\frac{n}{|X_i|})^{1/2}<t$ for $i=1,\dots,t$.  
\end{lemma}

\begin{proof}
	We prove that we can choose $\epsilon=\frac{1}{500}$ and $\delta=\frac{1}{100}$. Let $\prec$ be the partial ordering of the underlying poset of $P$.

	 Let $J=J_{0}=\lfloor \log_2 \epsilon n\rfloor+1$. For $j=1,\dots,J$, let $t_j=n^{1/2}2^{j/2}$. Then \begin{equation}\label{equ1}
	\sum_{i=1}^{J_0}t_i=\sum_{i=1}^{J_0} n^{1/2} 2^{i/2}\leq 2n\epsilon^{1/2}\frac{1}{1-2^{-1/2}}<\frac{n}{4}.
	 \end{equation}
	 Also, let $A'=\emptyset$ and $B'=\emptyset$. In what comes, we define an algorithm, which we shall refer to as the \emph{main algorithm}, which will find and output the desired $t$ and the $t$ sets $X_1,\dots,X_t$. During each step of the algorithm, we will  make the following changes: we will move certain elements of $A$ into $A'$, move certain elements of $B$ into $B'$, and decrease $J$. We think of the elements of $A'$ and $B'$ as ``leftovers''. We will keep track that at the end of each step of the algorithm, the following properties are satisfied:
	 \begin{enumerate} 
	 	\item	$|A|+|A'|=|B|+|B'|=n$,
	 	\item  $\displaystyle|A'|,|B'|\leq 2\sum_{i=J+1}^{J_{0}} t_i,$
	 	\item  for every $v\in B$, $|N(v)\cap A|<2^{J}$.
	 \end{enumerate}
	Note that by (\ref{equ1}) and property 1. and 2., we have $|A|,|B|\geq \frac{n}{2}$. Also, these properties are certainly satisfied in the beginning of the algorithm. Now let us describe a general step of our main algorithm.
	
    \bigskip	
	\noindent
	\textbf{Main algorithm:}
	
	For $i=1,\dots,J$, let $V_i$ be the set of vertices $v\in B$ such that $2^{i-1}\leq |N(v)\cap A|<2^{i}$, and let $V_{0}$ be the set of vertices $v\in B$ such that $N(v)\cap A=\emptyset$. Then by property 3., $B=\bigcup_{i=0}^{J} V_{i}$. 
	
	Let $1\leq k\leq J$ be maximal such that  $t_k<|V_k|$. First, consider the case if there exists no such $k$. Then $$n-\sum_{i=J+1}^{J_0}t_{i}-|V_0|\leq n-|B'|-|V_0|=|B|-|V_0|=\sum_{i=1}^{J} |V_i|\leq \sum_{i=1}^{J}t_i,$$
	where the first inequality follows from property 2., and the first equality is the consequence of property 1.. Comparing the left and right hand side, and using (\ref{equ1}), we get $|V_{0}|\geq n/2$. In this case, stop the algorithm and output  $t=2$, $X_1=V_{0}$ and $X_2=A$. Note that $\delta(\frac{n}{|X_i|})^{1/2}<t$ is satisfied for $i=1,2$.
	
	 Now suppose that there exists such a $k$.  Remove the elements of $V_i$ for $i>k$ from $B$, and add them to $B'$. Also, set $J:=k$. Then we added at most $ \sum_{i=k+1}^{J} t_i$ elements to $B'$, and properties 1.-3. are still satisfied.
	 
	 Now we shall run a \emph{sub-algorithm}. Let $W_0=V_k$, then with help of the sub-algorithm we construct a sequence $W_0\supset\dots\supset W_r$ satisfying the following properties. During each step of the sub-algorithm, we either find our desired $t$ and $t$ sets $X_1,\dots,X_t$, or we will move certain elements of $A$ to $A'$. At the end of the $l$-th step of this algorithm, $W_l$ be the set of vertices in $B$ that still has at least $2^{k-1}$ neighbors in $A$. We stop the algorithm if $W_l$ is too small.
	 
	 \bigskip
	 
	 \noindent
	 \textbf{Sub-algorithm:}
	 
	   Suppose that $W_{l}$ is already defined. If $W_{l}<2t_k$, then let $r=l$, stop the sub-algorithm, remove the elements of $W_l$ from $B$ and add them to $B'$. Make the update $J:=k-1$, and move to the next step of the main algorithm. Note that $B'$ satisfies property 2. Later, we will see that all the other properties are satisfied.
	 
	   On the other hand if $|W_l|\geq 2t_k$, we define $W_{l+1}$ as follows. Let $x_l=\frac{|W_l|}{t_k}$. Say that a vertex $v\in A$ is \emph{heavy} if 
	  $$|N(v)\cap W_l|\geq \frac{x_l 2^k}{t_k}|W_l|=\left(\frac{|W_l|}{t_k}\right)^{2}2^k=\frac{|W_l|^{2}}{n}=:\Delta_l,$$ and let $H_l$ be the set of heavy vertices. Counting the number of edges $f$ between $H_l$ and $W_l$ in two ways, we can write 
	  $$|H_l|\Delta_l\leq f< |W_l|2^{k},$$
	  which gives $|H_l|<\frac{t_k}{x_l}$. Remove the elements of $H_l$ from $A$ and add them to $A'$. Examine how the degrees of the vertices in $W_l$ changed, and consider the following two cases:
	
	\begin{description}
		\item[Case 1.] At least $\frac{|W_l|}{2}$ vertices in $W_l$ have at least $2^{k-1}$ neighbors in $A$. 
		
		Let $T$ be the set of vertices in $W_l$ that have at least $2^{k-1}$ neighbors in $A$, so $|T|\geq \frac{|W_l|}{2}$. Pick each element of $A$ with probability $p=2^{-k}$, and let $S$ be the set of selected vertices. Say that $v\in T$ is \emph{good}, if $|N(v)\cap S|=1$, and let $Y$ be the set of good vertices. Then
		$$\mathbb{P}(v\mbox{ is good})=|N(v)\cap A|p(1-p)^{|N(v)\cap A|-1}\geq \frac{1}{2}(1-2^{-k})^{2^{k}}\geq \frac{1}{6},$$ so $\mathbb{E}(|Y|)\geq \frac{|T|}{6}\geq \frac{|W_l|}{12}$. Therefore, there exists a choice for $S$ such that $|Y|\geq \frac{|W_l|}{12}$, let us fix such an $S$. For each $v\in S$, let $Y_v$ be the set of elements $w\in Y$ such that $N(w)\cap S=\{v\}$.  The important observation is that if $v,v'\in S$ and $v\neq v'$, then there is no edge between $Y_v$ and $Y_{v'}$.  Indeed, otherwise, if $w\in Y_v$ and $w'\in Y_{v'}$ such that $w\prec w'$, then $v\prec w\prec w'$, which means that $\{v,v'\}\in N(w')\cap S$, contradicting that $w'$ is good. Also, note that $$|Y_v|\leq |N(v)\cap W_l|\leq \min\{\epsilon n,\Delta_l\}=:\Delta_l'.$$ In other words, the sets $Y_v$ for $v\in S$ partition $Y$ into sets of size at most $\Delta_l'$. Here, we have 
		$$\frac{|Y|}{\Delta_l'}\geq \frac{|W_l|}{12\Delta_l'}\geq \max\left\{\frac{n}{12|W_{l}|},\frac{|W_l|}{\epsilon n}\right\}.$$
		By the choice of $\epsilon$, the right hand side is always at least $6$. But then we can partition $S$ into $t\geq \frac{|Y|}{3\Delta_l'}\geq 2$ parts $S_1,\dots,S_t$ such that the sets $X_i=\bigcup_{v\in S_i} Y_{v}$ have size at least $\Delta_l'$ for $i=1,\dots,t$. The resulting sets $X_1,\dots,X_t$ satisfy that there are no edges between $X_i$ and $X_j$ for $1\leq i<j\leq t$ and  
		$$t\geq  \frac{|Y|}{3\Delta_l'}\geq  \frac{n}{36|W_{l}|}\geq \frac{1}{36}\left(\frac{n}{\Delta_l}\right)^{1/2}\geq \frac{1}{36}\left(\frac{n}{|X_i|}\right)^{1/2}.$$
		Stop the main algorithm, and output $t$ and $X_1,\dots,X_t$. By the choice of $\delta$, this output satisfies our desired properties.
		
		\item[Case 2.] At most $\frac{|W_l|}{2}$ vertices in $W_l$ have at least $2^{k-1}$ neighbours in $A$.  
		
		In this case, define $W_{l+1}$ as the set of elements of $W_l$ with at least $2^{k-1}$ neighbors in $A$ (then $W_{l+1}$ is the set of all elements in $B$ with at least $2^{k-1}$ neighbors in $A$ as well). Also, move to the next step of the sub-algorithm.
	\end{description}
	
	Let us as check that if the main algorithm is not terminated, then in the end of the sub-algorithm, properties 1.-3. are still satisfied. 1. and 3. are clearly true, and 2. holds for $B'$. It remains to show that 2. holds for $A'$ as well. Note that as $|W_{l+1}|\leq \frac{|W_{l}|}{2}$ for $l=0,\dots,r-1$, and $|W_{r-1}|\geq 2t_k$, we have $|W_l|\geq 2^{r-l}t_k$ and $x_l\geq 2^{r-l}$. Compared to the first step of the sub-algorithm, $|A'|$ increased by 
	$$\sum_{l=0}^{r-1}|H_l|\leq \sum_{l=0}^{r-1}\frac{t_k}{x_l}\leq \sum_{l=0}^{r-1} \frac{t_k}{ 2^{r-l}}<t_{k}.$$
	Therefore, property 2. also holds.
	
	If the main algorithm was not stopped before and $J=0$, then stop the main algorithm, and output $t=2$ and $X_1=A, X_2=B$. Note that in this case there is no edge between $A$ and $B$, and $|A|,|B|\geq \frac{n}{2}$. By the choice of $\delta$, this output also satisfies our desired properties.
\end{proof}

Now we are ready to prove the main theorem of this section.

\begin{proof}[Proof of Theorem \ref{thm:poset}]
	Let $\epsilon,\delta>0$ be the constants given by Lemma \ref{lemma:main}. By Lemma \ref{lemma:sparse}, there exists $\beta>0$ such that the following holds. Let $<_l$ be a linear extension of the underlying partial order of $P$. Then either
	\begin{enumerate}
		\item there exists two disjoint subsets $A$ and $B$ of $P$ such that $A<_l B$, $|A|=|B|\geq\beta n$, $|N(v)\cap B|\leq \epsilon |B|$ for every $v\in A$, and  $|N(w)\cap A|\leq \epsilon |A|$ for every $w\in B$,  or
		\item there exist two disjoint sets $X_1$ and $X_2$ such that there are no edges between $X_1$ and $X_2$, and $|X_1|,|X_2|\geq \beta n$.
	\end{enumerate} 

   If 1. holds, then applying  Lemma \ref{lemma:main} to the comparability graph $P'=P[A\cup B]$, we get that there exist $t\geq 2$ and $t$ disjoint subsets $X_1,\dots,X_t$ of $P'$ such that there are no edges between $X_i$ and $X_j$ for $1\leq i<j\leq t$, and $\delta(\frac{|A|}{|X_i|})^{1/2}<t$ for $i=1,\dots,t$. Here, $\delta(\frac{|A|}{|X_i|})^{1/2}\geq \delta\beta^{1/2}(\frac{n}{|X_i|})^{1/2}.$ If $c$ is chosen sufficiently small with respect to $\beta$ and $\delta$, then $\delta\beta^{1/2}(\frac{n}{|X_i|})^{1/2}\geq (\frac{n}{|X_i|})^c$ holds whenever the left hand side is at least $2$, so this $c$ satisfies the desired properties.
   
   Now suppose that 2. holds. If $c$ is sufficiently small with respect to $\beta$, then the inequalities $t=2\geq (\frac{1}{\beta})^c\geq (\frac{n}{|X_i|})^c$ are satisfied for $i=1,2$. This finishes the proof.
\end{proof}

\section{The quasi-Erd\H{o}s-Hajnal property}\label{sect:EH}

Say that a family of graphs $\mathcal{G}$ has the \emph{quasi-Erd\H{o}s-Hajnal property}, if there exists a constant $c=c(\mathcal{G})$ such that the following holds for every $G\in \mathcal{G}$ with at least 2 vertices: there exist $t\geq 2$ and $t$ disjoint subsets $X_1,\dots,X_t$ of $V(G)$ such that $t\geq (\frac{|V(G)|}{|X_i|})^c$ for $i=1,\dots,t$, and either
\begin{enumerate}
	\item  $X_i$ is complete to $X_j$ for $1\leq i<j\leq t$, or
	\item there is no edge between  $X_i$ and $X_j$ for $1\leq i<j\leq t$.
\end{enumerate}

We show that the Erd\H{o}s-Hajnal property is actually equivalent to the quasi-Erd\H{o}s-Hajnal property in hereditary graph families. 

\begin{lemma}\label{lemma:qEH}
	If $\mathcal{G}$ is a hereditary family of graphs, then $\mathcal{G}$ has the Erd\H{o}s-Hajnal property if and only if it has the quasi-Erd\H{o}s-Hajnal property.
\end{lemma}

\begin{proof}
	If $\mathcal{G}$ has the Erd\H{o}s-Hajnal property, then there exists $c>0$ such that every $G\in \mathcal{G}$ contains a clique or an independent set of size at least $|V(G)|^c$. But then setting $t=|V(G)|^c$ and defining $X_1,\dots,X_t$ to be the single element sets formed by the vertices of such a clique or independent set shows that $\mathcal{G}$ also has the quasi-Erd\H{o}s-Hajnal property. It remains to show the other direction.
	
	Suppose that $\mathcal{G}$ has the quasi-Erd\H{o}s-Hajnal property. Let $G\in \mathcal{G}$ be a graph on $n$ vertices. Let $\mathcal{X}=\{V(G)\}$ and let $H$ be the graph with vertex set $\mathcal{X}$ (that is, $H$ has exactly one vertex, namely $V(G)$). We repeat the following procedure until every element of $\mathcal{X}$ has only one vertex.  If $\mathcal{X}$ contains a set of size at least $2$, say $X\in \mathcal{X}$, then consider the induced subgraph $G[X]\in\mathcal{G}$. Then there exist $t\geq 2$ and  $t$ disjoint subsets $X_1,\dots,X_t$ of $X$ such that $t\geq (\frac{|X|}{|X_i|})^c$ for $i=1,\dots,t$, and either
	\begin{enumerate}
		\item $X_i$ is complete to $X_j$ for $1\leq i<j\leq t$, or
		\item there is no edge between  $X_i$ and $X_j$ for $1\leq i<j\leq t$.
	\end{enumerate}
	
	Remove the set $X$ from $\mathcal{X}$ and add the sets $X_1,\dots,X_t$. Also, if 1. happens, replace the vertex $X$ in $H$ with a clique on $\{X_1,\dots,X_t\}$, otherwise, replace $X$ in $H$ with an independent set on $\{X_1,\dots,X_t\}$. More precisely, $X_i$ has the same neighborhood as $X$ had outside of $\{X_1,\dots,X_t\}$, and $\{X_1,\dots,X_t\}$ induces either a clique or an independent set depending on whether 1. or 2. holds, respectively.
	
	Note that $\sum_{i=1}^{t}|X_i|^{c}\geq |X|^{c}$, therefore the sum $\sum_{Y\in \mathcal{X}} |Y|^{c}$ did not decrease after the change. Thus, we have $\sum_{Y\in \mathcal{X}} |Y|^{c}\geq n^c$ in each step of the procedure. This implies that at the end of the procedure, that is, when every element of $\mathcal{X}$ is a single vertex set, we have $|\mathcal{X}|\geq n^{c}$.
	
	Moreover, at each step of the procedure, the graph $H$ is a cograph. It is well known that cographs are perfect, therefore, at the end of the procedure, either $H$ or its complement contains a clique of size at least $n^{c/2}$. This clique corresponds to a clique or an independent set of size at least $n^{c/2}$ in $G$. As this is true for every $G\in \mathcal{G}$, $\mathcal{G}$ has the Erd\H{o}s-Hajnal property.
\end{proof}

\section{String graphs}\label{sect:strings}

In this section, we put all the ingredients together to prove Theorem \ref{thm:main}.

A \emph{separator} in a graph $G$ is a subset $S$ of the vertices such that after the removal of $S$, every connected component of $G$ has size at most $\frac{2|V(G)|}{3}$. It was proved by Fox and Pach \cite{FP08_sep0} that if $G$ is the intersection graph of a family of $n$ curves and $g$ is the total number of crossings between the curves, then $G$ contains a separator of size $O(\sqrt{g})$. Later, Fox and Pach  \cite{FP10_sep1} showed that if $G$ is a string graph with $m$ edges, then it contains a separator of size $O(m^{3/4}\sqrt{\log m})$, and proposed the conjecture that one can also find a separator of size $O(\sqrt{m})$, which is then optimal up to the constant factor. In \cite{FP10_sep1,FP14_sep2}, Fox and Pach also provide a number of applications of the existence of small separators. The size of the smallest separator was improved to $O(\sqrt{m}\log m)$ by Matou\v sek \cite{M14}, and recently, Lee \cite{L17} completely settled the aforementioned conjecture of Fox and Pach. The result of Lee immediately implies the following lemma, which will be the first key ingredient in our proof.

\begin{lemma}\label{lemma:separator}
	There exists a constant $\lambda>0$ such that the following holds. If $G$ is a string graph with $n$ vertices and at most $\lambda n^2$ edges, then there exist two disjoint subsets $X_1$ and $X_2$ of the vertices such that there are no edges between $X_1$ and $X_2$, and $|X_1|=|X_2|\geq \lambda n$.
\end{lemma}

Let us remark that the author of this paper \cite{T19} proved the following sharpening of this lemma: If the edge density of a string graph is below $1/4$, then one can find two linear sized subsets of the vertices with no edges between them. However, there are string graphs with edge density arbitrarily close to $1/4$ which only contain logarithmic sized such sets.

The final ingredient we need for our proof is the following result of Fox and Pach \cite{FP12_inc}, which tells us that every dense string graph contains a dense incomparability graph.

\begin{lemma}\label{lemma:incomp}
	For every $\lambda>0$ there exist $\epsilon>0$ such that the following holds. If $G$ is a string graph with $n$ vertices and at least $\lambda n^{2}$ edges, then $G$ contains a subgraph $G'$ with  $V(G')=V(G)$ such that $G'$ is an incomparability graph with at least $\epsilon n^2$ edges.
\end{lemma}

By Lemma \ref{lemma:qEH}, in order to prove Theorem \ref{thm:main}, it is enough to show that the family of string graphs has the quasi-Erd\H{o}s-Hajnal property. This almost immediately follows from a combination of the results discussed in this paper.

\begin{theorem}
	The family of string graphs has the quasi-Erd\H{o}s-Hajnal property.
\end{theorem}

\begin{proof}
	 Let $\lambda$ be the constant given by Lemma \ref{lemma:separator}, and let $\epsilon$ be the constant given by Lemma \ref{lemma:incomp} (with respect to $\lambda$). Also, let $c_0$ be the constant $c$ given by Theorem \ref{thm:incomparability} with respect to $\epsilon$. We show that the family of string graphs has the quasi-Erd\H{o}s-Hajnal property with the exponent $$c=\min\left\{c_0,\frac{1}{\log_2 (1/\lambda)}\right\}.$$
	 
	 Let $G$ be a string graph with $n$ vertices. If $G$ has at most $\lambda n^2$ edges, then $G$ contains two disjoint subsets $X_1$ and $X_2$ with no edges between them such that $|X_1|=|X_2|\geq \lambda n$. Setting $t=2$, we have $t\geq  (\frac{1}{\lambda})^{c}\geq (\frac{n}{|X_i|})^{c}$ for $i=1,2$. 
	 
	 Now suppose that $G$ has more than $\lambda n^2$ edges, then $G$ contains an incomparability graph $G'$ with at least $\epsilon n^2$ edges.  Then, by Theorem \ref{thm:incomparability},  there exist $t\geq 2$ and $t$ disjoint subsets $X_1,\dots,X_t$ of $G'$ such that $X_i$ is complete to $X_j$ for $1\leq i<j\leq t$, and $(\frac{n}{|X_i|})^{c}<t$ for $i=1,\dots,t$. This finishes the proof.
\end{proof}

\section{Acknowledgements}

We would like to thank Jacob Fox and J\'anos Pach for their valuable remarks on the presentation of this paper. 

The author was also partially supported by MIPT and the grant of Russian Government N 075-15-2019-1926.

\end{document}